\documentclass[12pt]{article}

\usepackage{amsmath,amssymb,amsthm,mathrsfs,amsfonts,amscd,esint,cite}

\usepackage{sidecap}
\usepackage{float}
\usepackage{extarrows}
\usepackage{verbatim}
\usepackage{hyperref}
\usepackage[usenames,dvipsnames]{xcolor}
\usepackage{enumitem}
\usepackage{mathtools}

\newtheorem{thm}{Theorem}[section]

\newtheorem{lemma}[thm]{Lemma}
\newtheorem{cor}[thm]{Corollary}
\newtheorem{prop}[thm]{Proposition}

\newtheorem{remark}[]{Remark}

\def\XXint#1#2#3{{\setbox0=\hbox{$#1{#2#3}{\int}$}
		\vcenter{\hbox{$#2#3$}}\kern-.5\wd0}}

\def \R {\mathbb R}

\def \RN  {\mathbb{R}^N}
\def \RD  {\mathbb{R}^d}

\numberwithin{equation}{section}

\begin{document}
	
\title{Decay of solutions to one-dimensional inhomogeneous nonlinear Schr\"odinger equations\thanks{This work is supported by NSFC under the grant Nos. 12075102, 12301090 and 11971212.} }
\author{Zhi-Yuan Cui, Yuan Li and Dun Zhao\thanks {E-mail addresses: zhaod@lzu.edu.cn (Dun Zhao)}\\{\small School of Mathematics and Statistics, Lanzhou University, Lanzhou 730000, China}}
\date{}
\maketitle
	
\begin{abstract}

We investigate the decay estimates of global solutions for a class of one-dimensional inhomogeneous nonlinear Schr\"odinger equations. While most existing results focus on spatial dimensions $d\geq2$, the decay properties in one dimension remain less explored due to the absence of effective Morawetz inequalities. For equations without external potential, by establishing a localized Virial-Morawetz identity, we derive decay estimates in the context of the $L^r$-norm for global solutions within a compact domain as a time subsequence approaches infinity. This decay result can be applied to obtain a criterion for energy scattering. Additionally, by establishing another type of Virial-Morawetz identity under more strict conditions, we demonstrate the decay result for odd solutions for any time sequence that approaches infinity. Utilizing some results about bound states proved by Barry Simon, we also show that similar decay results hold for the global odd solutions of equations with suitable external potentials that contain inverse power type and Yukawa-type potentials.
		
\vspace{0.5cm}

\noindent \textbf{Key words}: inhomogeneous nonlinear Schr\"{o}dinger equation; decay; Morawetz-type estimate; weighted Sobolev space; external potential

\vspace{0.5cm}
		
\noindent \textbf{MSC Classification}: 35Q55, 35B40
	
\end{abstract}

\section{Introduction}
	
In this paper, we explore the decay properties of global solutions of the inhomogeneous nonlinear Schr\"odinger (INLS) equation in the case $d=1$:
	\begin{equation}\label{eq1}
		\left\{\begin{aligned}
			(i\partial_t+\Delta) u(t,x)&=K(x)|u(t,x)|^{2\sigma}u(t,x)+\mu V(x)u(t,x),\\
			u(0,x)&=u_0(x)\in H^1(\RD),
		\end{aligned}\right.~~(t,x)\in\mathbb{R}\times\mathbb{R}^d,
	\end{equation}
where $0<\sigma<\infty$ and functions $K$, $V: \RD\rightarrow\R$ are not constant. The INLS equation has many applications in various physical contexts. For instance, it models beam propagation, where a preliminary laser beam creates a channel with reduced electron density in a plasma, enabling stable high-power propagation by reducing the nonlinearity inside the channel \cite{Gill2000,LT1994}. In this case, the complex-valued function $u=u(t,x)$ denotes the electric field in laser optics, $K$ is proportional to the electric density, and $V$ represents the external potential. When $\sigma=1$, it becomes the Gross-Pitaevskii equation, which describes the dilute Bose-Einstein condensate with the spatially dependent two-body interactions \cite{2007PRL}. If $\mu=0$ and $K(x)=\pm1$, equation \eqref{eq1} becomes the standard model, which has been extensively studied in  previous years. We refer to, for instance, \cite{FXC2011,Cz2003,Dodson2019,DHR2008,HR2008} and the references therein.

After the pioneering work by Merle \cite{Merle1996} under the assumption $\mu=0$ and $k_1<K(x)<k_2$ with $k_1,~ k_2>0$, the INLS equation has attracted particular attention over the past two decades. For example, Rapha\"el and Szeftel \cite{RS2011} gave a necessary and sufficient condition on the inhomogeneity $K(x)$ to ensure the existence of minimal blow-up solutions. In particular, many works have been devoted to the case that $K(x)\sim|x|^{-b}$, and the topics cover many subjects, such as the well-posedness theory, the existence and stability of solitary waves, and blow-up solutions. For example, we refer to \cite{BF2005,GS2008,BL2023,Farah2016,Guz2017} and the references therein.

Now we recall some scattering and decay results for the INLS equation. When $\mu=0$, $K(x)=-|x|^{-b}$ (focusing case), the energy scattering (decay) for the problem \eqref{eq1} has been initially established by Farah et al. \cite{FG2017} for $0<b<\frac{1}{2},~ \sigma=1$ and $d=3$ with radial assumption, the proof is based on the concentration-compactness argument developed by Kenig and Merle \cite{KM2006}. This result has been extended to $d\geq2$ for $0<b<\min\{\frac{d}{3},1\}, ~\frac{4-2b}{d}<2\sigma<2_*$, see \cite{FG2019}. Later, the energy scattering of the non-radial solution of \eqref{eq1} for $0<b<\frac{1}{2},~\sigma=1$ and $d=3$ was obtained by Miao, Murphy, and Zheng \cite{MMZ2021}. Recently, this result has been extended to $d \geq 2$ and more large ranges of the parameters $\sigma$ and $b$, refer to \cite{CFGM2020,Dinh2021-3,CC2021}. In addition, for $d=3$, the energy scattering for the solutions of the energy-critical INLS equations was proved \cite{CHL2020,GM2021}. When $\mu=0$, $K(x)=|x|^{-b}$ (defocusing case), the energy scattering (decay) for $d\geq3$ has been proved by Dinh \cite{Dinh2017} by the same arguments as in \cite{Vis2008} via Morawetz estimates:
	$$		\int_J\int_{\RD}|x|^{-b-1}|u(t,x)|^{2\sigma+2}dxdt<\infty.	$$
After that, Dinh \cite{Dinh2021-2} proved energy scattering for $d\!=\!2$ with the radial assumption for both focusing and defocusing cases via similar Morawetz estimates. More recently, \cite{AT2024} tackled the scattering problem with small initial data under more general conditions on $K(x)$. In particular, when $|K(x)|=|x|^{-b}(1+|x|^2)^{-\frac{a}{2}}$, the scattering result has been yielded for the case $d\geq4$. As we can see, the majority of studies have concentrated on spatial dimensions $d\geq2$. When $\mu\neq0$,  Dinh \cite{Dinh2021-4} has considered the global dynamics for both focusing and defocusing INLS equations with Kato-type potentials for $d=3$ with radially symmetric initial data. \cite{CG2021} has considered the INLS equations with inverse-square potential for $d\geq 3$.

Notice that all the previously mentioned works do not study the decay properties of solutions in the one-dimensional case, one reason is due to the absence of effective Morawetz inequalities. To our knowledge, there are two predominant methods for ascertaining the decay property of NLS equations in the energy space: the classical Morawetz technique \cite{GV1985,Cz2003} and the interaction Morawetz approach \cite{CGT2009}. However, due to the non-conservation of momentum in the INLS problem  \eqref{eq1}, the tensor products method for one dimension, as proposed in \cite{CGT2009}, remains ambiguous for INLS equations. Additionally, the classical Morawetz estimation techniques from \cite{GV1985,Cz2003} become invalid due to the negativity of the quadratic form when spatial dimensions $d\leq2$. Recently, based on an idea of \cite{KMM2017}, Mart\'{\i}nez \cite{M2020No} considered the long-time asymptotics of odd solutions of the nonlinear Schr\"odinger equation with semi-linear and Hartree nonlinearities in the space of one dimension and proved that the solutions decay to zero in compact regions of space as time tends to infinity. Inspired by the ideas in \cite{M2020No} and \cite{CC2021}, this paper aims to investigate the decay estimates for the solutions of the INLS equation \eqref{eq1} in the one-dimensional context. We give a Morawetz-type estimate and demonstrate that, in a localized manner, the decay of (odd) solutions of equation (\ref{eq1}) (without or with an external potential) for $d=1$, which is achieved by choosing  suitable weights to establish Virial-Morawetz identities.
	
Hereafter, we assume that  $0\!<\!b\!<\!1$, and $K\in C^1(\R\setminus \{0\})$ is an even function that satisfies one of the following conditions:
\begin{itemize}[leftmargin=4em]
		\item[($K_1$)]  $ 0<K(x)\lesssim|x|^{-b},\quad xK^{\prime}(x)\leq0,\quad |K^{\prime}(x)|\lesssim|x|^{-b-1}$,
\end{itemize}
\begin{itemize}[leftmargin=4em]
		\item[($K_2$)] $ |K(x)|\lesssim|x|^{-b},\quad |K^{\prime}(x)|\lesssim|x|^{-b-1}$,
\end{itemize}
\begin{itemize}[leftmargin=4em]
		\item[($K_3$)] $ |K(x)|\lesssim|x|^{-b}(1+|x|)^{-2+b},\quad |K^{\prime}(x)|\lesssim|x|^{-b-1}(1+|x|)^{-2+b}$,
\end{itemize}
\begin{itemize}[leftmargin=4em]
     \item [($K_4$)] $ |K(x)|\lesssim|x|^{-b}(1+|x|)^{-3},\quad |K^{\prime}(x)|\lesssim|x|^{-b-1}(1+|x|)^{-3}$.
\end{itemize}
	
 Under the assumption $(K_1)$, by an abstract theory presented in \cite{Cz2003}, or under the assumption $(K_2)$, or $(K_3)$, or $(K_4)$, through the method proposed by Farah\cite{Farah2016}, the global well-posedness of equation \eqref{eq1} can be achieved when $\mu=0$. We have
\begin{prop}
Assume $\mu=0$ and $u_0(x)\in H^1(\R)$, then
\begin{itemize}
			\item [(1)] If $K(x)$ satisfies $(K_1)$ and $0<\sigma<\infty$, then problem \eqref{eq1} is globally well-posed in $H^1(\R)$.
			\item[(2)] If $K(x)$ satisfies $(K_2)$, or $(K_3)$, or $(K_4)$, then problem \eqref{eq1} is globally well-posed in $H^1(\R)$ for $0<\sigma<2-b$ or for $2-b\leq\sigma<\infty$ with small initial data.
\end{itemize}
\end{prop}
	
Our first result concerns INLS without potential.
\begin{thm}\label{theorem{1}}
Assume that $\mu=0$. Let $u(t)\in H^1(\R)$ be a global solution of \eqref{eq1} with initial data $u_0\in H^1(\R)$. Then, the following holds:
\begin{itemize}
\item[(1)] For $0<\sigma<\infty$,  if $K$ satisfies $(K_1)$, then there exists a sequence of times $t_n\rightarrow\infty$ such that, for any bounded interval $I\subset\R$,
\begin{equation}\label{eq2}
				\lim_{n\rightarrow\infty}(\|u(t_n)\|_{L^2(I)}+\|u(t_n)\|_{L^\infty(I)})=0.
\end{equation}
Moreover, one has the Morawetz-type estimate
\begin{equation}\label{eq3}
	\frac{1}{T}\int_{0}^{T}\int_{\R}\frac{|xK^{\prime}(x)|}{1+|x|}|u|^{2\sigma+2}dxdt<\frac{R}{T}+\frac{1}{R^2}
\end{equation}
for any $T>0$ and $R>0$.
\item[(2)] For $2-b\!<\!\sigma\!<\!\infty$, if $K$ satisfies $(K_2)$, and there exists a small $\epsilon>0$ such that $\|u_{0}\|_{H^1(\R)}\leq\epsilon$, then there exists a sequence of times $t_n\rightarrow\infty$ such that for any bounded interval $I\subset\R$, \eqref{eq2} holds. Additionally, one has the Morawetz-type estimate
\begin{equation}\label{eq1.5}
		\frac{1}{T}\int_{0}^{T}\int_{|x|\leq R}|x|^{-b}|u|^{2\sigma+2}dxdt\lesssim\frac{R}{T}+\frac{1}{R^b},
\end{equation}for any $T>0$ and sufficiently large $R>0$.
\item [(3)] For $0\!<\!\sigma\!\leq\!2-b$, if $K$ satisfies $(K_3)$ and  $\|u_{0}\|_{H^1(\R)}\leq\epsilon$ for some small $\epsilon>0$, then there exists a sequence of times $t_n\rightarrow\infty$ such that for any bounded interval $I\subset\R$, \eqref{eq2} holds. Additionally, one has the Morawetz-type estimate
\begin{equation}\label{eq1.2.3}
\frac{1}{T}\int_{0}^{T}\int_{|x|\leq R}|x|^{-b}(1+|x|)^{-2+b}|u|^{2\sigma+2}dxdt\lesssim\frac{R}{T}+\frac{1}{R^2},
\end{equation}for any $T>0$ and sufficiently large $R>0$.
\end{itemize}
\end{thm}
Although Theorem \ref{theorem{1}} only provides the decay estimate for a subsequence of times $t_n\rightarrow\infty$, it can be applied to some scattering problems, that is, under certain conditions, we can apply it to  exclude the non-scatting solutions (the details can be found at the end of Sect. 3.1).

 Note that if the initial data $u_0$ is odd, then the solution $u(t,x)$ of \eqref{eq1} is also odd. In particular, we have $u(t,0)\equiv0$.  Using this property, we have the following strengthened result.

 \begin{thm}\label{T1}
     	Assume $\mu=0$, $0<\sigma<\infty$, and $u(t)\in H^1(\R)$ is a global odd solution of \eqref{eq1} with odd initial data $u_0\in H^1(\R)$. We have
		\begin{itemize}
			\item[(1)]  If $K$ satisfies $(K_1)$, then for any bounded interval $I\subset\R$,
			\begin{equation}\label{eq decay}
				\lim_{t\rightarrow\infty}(\|u(t)\|_{L^2(I)}+\|u(t)\|_{L^\infty(I)})=0.
			\end{equation}
       Moreover, one has the Morawetz-type estimate
			\begin{equation}\label{eq Mod}
				\int_{-\infty}^{\infty}\int_{\R}\frac{|xK^{\prime}(x)|}{1+|x|}|u|^{2\sigma+2}dxdt<\infty.
			\end{equation}
			\item[(2)]   If $K$ satisfies $(K_4)$ and there exists a small $\epsilon>0$  such that $\|u_{0}\|_{H^1(\R)}\leq\epsilon$, then for any bounded interval $I\subset\R$, \eqref{eq decay} holds, and one has the Morawetz-type estimate
			\begin{equation}\label{eq Mof}
				\int_{-\infty}^{\infty}\int_{\R}|x|^{-b}(1+|x|)^{-4}|u|^{2\sigma+2}dxdt<\infty.
			\end{equation}
		\end{itemize}
\end{thm}

\begin{remark}~
  \begin{itemize}
		 \item[(1)]   Notice that $|x|^{-b}$ and $-|x|^{-b}$ satisfy $(K_1)$ and $(K_2)$, respectively. Furthermore, $(K_2)$, $(K_3)$, and $(K_4)$ allow the change of sign of $K(x)$.

         \item[(2)] Although Theorem \ref{theorem{1}} and Theorem \ref{T1} give some decay results for $d\!=\!1$, the scattering of the solution in $H^1(\R)$ is still an unresolved problem. In \cite{AK2021},  some scattering results for one-dimensional INLS equations (with $|K(x)|=|x|^{-b}$  and $2-b<\sigma$) in $H^s$ with $0<s<1$ and $0<b<1-s$ have been obtained, but $s=1$ is not accessible.
		\end{itemize}
 \end{remark}
	
	\vspace{0.3cm}
	
 The following results concerns the INLS equation \eqref{eq1} with potential.

 \begin{thm}\label{theorem{2}}
Assume that  $V\in C^1(\R\setminus \{0\})$ is an even function and satisfies

\hspace{1cm} $(V)\qquad\qquad  \int_\R |xV^{\prime}(x)|(1+|x|)^2dx < \infty.$

\noindent Let $u\in C(\R;H^1(\R))$ be a global solution of \eqref{eq1} with odd initial data $u_0\in H^1(\R)$, where $\|u_0\|_{H^1(\R)}\leq\epsilon$ for some small $\epsilon>0$. Then, there exists a constant $\mu_0>0$ such that for any $\mu\in(0,\mu_0)$ and bounded interval $I\subset\R$,
\begin{equation}\label{eqv1}
			\liminf_{t\rightarrow\infty}(\|u(t)\|_{L^2(I)}+\|u(t)\|_{L^\infty(I)})=0,
\end{equation}
provided that one of the following conditions is satisfied:

(1) $2-b<\sigma<\infty$ and $K$ satisfies $(K_2)$ or $0<\sigma\leq 2-b$ and $K$ satisfies $(K_3)$.

(2) $0<\sigma<\infty$ and $K$ satisfies $(K_1)$ or $(K_4)$.

\noindent In particular, if (2) holds, \eqref{eqv1} can be strengthened to
\begin{equation}\label{eqv2}
			\lim_{t\rightarrow\infty}(\|u(t)\|_{L^2(I)}+\|u(t)\|_{L^\infty(I)})=0.
		\end{equation}
\end{thm}

\begin{remark}~
\begin{itemize}
\item[(1)] A noteworthy emphasis of Theorem \ref{theorem{2}} is freedom from imposing specific spectral properties on the operator	$-\partial_x^2\pm\mu V$. Instead, it relies solely on the hypothesis $(V)$, evading non-resonance conditions. This allows for the realization of \eqref{eqv1} and \eqref{eqv2} by some bound state of Schr\"{o}dinger operators \cite{Klaus1977, Simon1976}.		
			
\item[(2)]  The external potentials with condition $(V)$, contain the inverse power forms like $$V(x)=|x|^{-m}(1+|x|)^{-n},~1>m\geq 0,~n+m>2,$$
			and the Yukawa-type potential
			$$V(x)=|x|^{-m}e^{-n|x|},~m\in(0,1),~n>0.$$
In comparison, Mart\'{\i}nez \cite{M2020No} has dealt with a Schwartz even function, and Dinh \cite{Dinh2021-4} has studied the energy scattering for both focusing and defocusing INLS equations with Kato-type potentials and radially symmetric initial data for $d=3$. Campos and Guzm\'an \cite{CG2021} investigated the INLS equations with inverse-square potential for $d\geq3$.
\end{itemize}
\end{remark}

 The structure of this paper is as follows: Sect. 2 introduces the necessary notations and some preliminaries; Sect. 3 gives the proof of Theorem \ref{theorem{1}} and \ref{T1}; Sect. 4 provides the proof of Theorem \ref{theorem{2}}. 
	
\section{Preliminaries}
	
From now on, we will use the notation $A\lesssim B$ to mean $A\leq CB$ for some constant $C > 0$, and $C$ may change from line to line. $\Re z$ will denote the real part of the complex number $z$ and $\Im z$ the imaginary part. For a complex function $u$, we will denote $\Re u$ by  $u_1$ and $\Im u$ by  $u_2$. According to the context, $f^{\prime}$ or $f_x$ will denote the derivative of $f$ with respect to $x$ and $f^{(\alpha)}$ will denote the $\alpha$-th derivative of $f$ with respect to $x$. The subspace $H^1_{odd}(\R)$ is defined as
	$$H^1_{odd}(\R)=\{u\in H^1(\R)~|~u~ is~ odd\}.$$
Let $\alpha(x) \geq 0$ be a function, we set
\begin{equation*}
		\|u(t)\|^2_{L^2_\alpha(\R)}:= \int_{\R}\alpha(x)|u(t,x)|^2dx
\end{equation*}
and
\begin{equation*}
		\|u(t)\|^2_{H^1_\alpha(\R)}:= \int_{\R}\alpha(x)(|u_x(t,x)|^2+|u(t,x)|^2)dx
\end{equation*}
to define the $L^2$-  and $H^1$-norm with weight $\alpha(x)$, respectively.
\begin{lemma}[\cite{Simon1976},  \cite{Klaus1977}]\label{lemma2.1}
Let $V_0$ be a nonzero potential that satisfies
\begin{equation}\label{eq5}
			\int_\R(1+|x|)|V_0(x)|dx<\infty,
\end{equation}
and $\mu>0$ is sufficiently small. 	Then, $-\frac{d^2}{dx^2}+\mu V_0$ has a unique negative eigenvalue if and only if \begin{equation*}
			\int_\R V_0(x)dx \,\leq0.
\end{equation*}
Moreover, if $\int_\R V_0(x)>0$, then  $-\frac{d^2}{dx^2}+\mu V_0$ has no negative eigenvalue.
\end{lemma}
\begin{remark}
Simon \cite{Simon1976} firstly introduced the condition  $\int_\R(1+x^2)|V_0(x)|dx<\infty$ to guarantee the result in Lemma \ref{lemma2.1}, this condition was improved to \eqref{eq5} by Klaus \cite{Klaus1977}. In addition, if $V_0$ is an even function, then such a unique negative eigenvalue is associated with an even eigenfunction.
\end{remark}

 Now, we list the following Gagliardo-Nirenberg inequalities, which can be found in Theorem 1.2 and Remark 1.3 of \cite{Farah2016}.
 \begin{lemma}[Gagliardo-Nirenberg inequality]
		For $0<b<1$ and $u\in H^1(\R)$, the following inequalities hold:
	\begin{equation}\label{eq6}
			\int_{\R}|x|^{-b}|u(x)|^{6-2b}dx \leq C_{GN} \| u_x\|_{L^2(\R)}^{2}\|u\|_{L^2(\R)}^{4-2b},
	\end{equation}
   \begin{equation}\label{eq6.1}
      \int_{\R}|x|^{-b}|u(x)|^{2}dx \lesssim \| u_x\|_{L^2(\R)}^{2}+\|u\|_{L^2(\R)}^{2} .
  \end{equation}
\end{lemma}
 
Let $w\in H^1(\R)$ be a real-valued function. Define a quadratic form:
	\begin{equation*}
		B(w)=-\frac{1}{2}\int_{\R}\frac{6}{(1+|x|)^4}|w|^2dx+2\int_{\R}\frac{1}{(1+|x|)^2}|w_x|^2dx.
	\end{equation*}
	\begin{lemma}\label{lem3.2}
		Let $u=u_1+iu_2~ and~ u_1, u_2 \in H^1_{odd}(\R)$, then for $\alpha(x)=\frac{1}{(1+|x|)^4}$, there is a constant $C>0$ such that
		\begin{equation}\label{eq12}
			\|u\|^2_{H^1_\alpha(\R)}\leq C(B(u_1)+B(u_2)).
		\end{equation}
	\end{lemma}
\begin{proof}  Let $w=u_k,~k=1,2$. We start with the following identity
\begin{equation}\label{eq0}
			\int_{\R}[(\frac{1}{1+|x|}w)_x]^2dx\!=\!\int_{\R}\frac{1}{(1+|x|)^2}|w_x|^2dx
\!+\!\int_{\R}\frac{-2x}{(1\!+\!|x|)^3|x|}w_xwdx\!+\!\int_{\R}\frac{1}{(1+|x|)^4}|w|^2dx.
\end{equation}
Let us firstly deal with the second term on the right-hand side.
\begin{equation}\label{eq13}
\begin{split}
\int_{\R}\frac{-2x}{(1+|x|)^3|x|}w_xwdx=\int_{-\infty}^{0}\frac{1}{(1+|x|)^3}\left(|w|^2\right)_xdx+\int_{0}^\infty\frac{-1}{(1+|x|)^3}\left(|w|^2\right)_xdx.
\end{split}
\end{equation}
Integrating by parts, as $|w(0)|^2=|u_k(0)|^2=0$, we get
\begin{equation*}
			\int_{\R}\frac{-2x}{(1+|x|)^3|x|}w_xwdx=-\int_{\R}\frac{3}{(1+|x|)^4}|w|^2dx.
\end{equation*}
Combining with \eqref{eq0}, one has
\begin{equation*}			\int_{\R}\frac{1}{(1+|x|)^2}|w_x|^2dx=\int_{\R}\left[\left(\frac{1}{1+|x|}w\right)_x\right]^2dx+\int_{\R}\frac{2}{(1+|x|)^4}|w|^2dx.
\end{equation*}
		By
		$$\int_{\R}\frac{1}{(1+|x|)^4}|w_x|^2dx\leq\int_{\R}\frac{1}{(1+|x|)^2}|w_x|^2dx,$$
		it yields that
		$$2	\int_{\R}\frac{1}{(1+|x|)^2}|w_x|^2dx-\frac{1}{2}\int_{\R}\frac{6}{(1+|x|)^4}|w|^2dx\gtrsim\int_{\R}\frac{1}{(1+|x|)^4}\left(|w|^2+|w_x|^2\right)dx,$$
which can conclude \eqref{eq12}.
\end{proof}

\section{INLS without potential}

In this section, we consider equation \eqref{eq1} for the case $\mu=0$, which reads
\begin{equation} \label{eq7}
\left\{
\begin{aligned}
			i\partial_t u(t,x) +\Delta  u(t,x) &=K(x)|u(t,x)|^{2\sigma}u(t,x),\\
			u(0,x)&=u_0(x)\in H^1(\mathbb{R}),
\end{aligned}\right.~~(t,x)\in\mathbb{R}\times\mathbb{R},
\end{equation}
We will complete the proof of Theorem \ref{theorem{1}} and \ref{T1}
 and give an application of Theorem \ref{theorem{1}} at the end of Sect. \ref{s3.1}.

\subsection{Proof of Theorem \ref{theorem{1}}}\label{s3.1}

	Following the ideas in \cite{CC2021}, we establish a localized Virial identity associated with \eqref{eq7} and incorporate the decaying factor in the nonlinearity to prove Theorem \ref{theorem{1}}. In what follows, we choose a smooth real-valued function $\phi~:~\R\rightarrow\R$ defined as
\begin{equation*}
\phi(x)=\left\{
\begin{aligned}
    & -R, &x<-R,\\
	& x, \qquad&|x|\leq\frac{R}{2},\\
    & R, &x>R,
\end{aligned}
     \right.
\end{equation*}
and require that in the region $\frac{R}{2}<|x|\leq R$, $\phi$ satisfies $\phi_x\geq0$ and $|\phi^{(\alpha)}(x)|\lesssim R|x|^{-\alpha}$ for $\alpha\geq1$.  Obviously, we have $\phi^{(\alpha)}(x)=0$ for $\alpha\geq1$ in the region $|x|>R$. Indeed, there are functions with the above properties, e.g. 
 \begin{equation*}
\phi(x):=\left\{
\begin{aligned}
    & -R, &x\leq -R,\\
    &-R+(R+x)\eta\left(\frac{2(R+x)}{R}\right), &-R<x<-\frac{R}{2}\\
	& x, \qquad&|x|\leq\frac{R}{2},\\
     &R-(R-x)\eta\left(\frac{2(R-x)}{R}\right), &\frac{R}{2}<x<R\\
    & R, &x\geq R,
\end{aligned}
     \right.
\end{equation*}
where
\begin{equation*}
    \eta(t):=\left\{
\begin{aligned}
    & 0 &t\leq 0,\\
	& \frac{e^{-\frac{1}{t}}}{e^{-\frac{1}{t}}+e^{-\frac{1}{1-t}}}, \qquad&0<t<1,\\
    & 1, &t\geq 1,
\end{aligned}
     \right.
\end{equation*}
and $\eta^{(\alpha)}$ are bounded depending on $\alpha$. By direct calculation, for $\alpha\geq 1$ and $\frac{R}{2}<|x|\leq R$ we have $$|\phi^{(\alpha)}(x)|\lesssim \frac{1}{R^{\alpha-1}}\left|\eta^{(\alpha-1)}\left(\frac{2(R-|x|)}{R}\right)\right|+\frac{(R-|x|)}{R^\alpha}\left|\eta^{(\alpha)}\left(\frac{2(R-|x|)}{R}\right)\right|\lesssim R|x|^{-\alpha}.$$
Assume that $u(t)\in H^1(\R)$ is a solution of equation \eqref{eq7}. Define
 \begin{equation}\label{eq8}
		I(u(t)):=\Im\int_{\R}\phi(x)u(t,x)\overline{u}_x(t,x)dx.
 \end{equation}
	
\begin{lemma}\label{lem3.1}
Let $u\in C(\R;H^1(\R))$ be a solution of equation \eqref{eq7}. Then $I(u(t))$ is well-defined and bounded in time. Moreover, the following Virial identity holds:
\begin{equation}\label{eq9}
  \begin{split}
				-\frac{d}{dt}I(u(t))=&2\int_{\R}\phi_x|u_x|^2dx-\frac{1}{2}\int_{\R}\phi_{xxx}|u|^2dx\\ &-\left(\frac{2}{2\sigma+2}-1\right)\int_{\R}\phi_x K(x)|u|^{2\sigma+2}dx-\frac{2}{2\sigma+2}\int_{\R}\phi K^{\prime}(x)|u|^{2\sigma+2}dx.
 \end{split}
 \end{equation}
  
\end{lemma}

\begin{proof}
Let $u(t)\in H^1(\R)$ be a solution of equation \eqref{eq7}. Integrating by parts, we get
		\begin{equation*}
			\begin{split}
				\frac{d}{dt}I(u(t))=&\Im\int_{\R}\phi u_t\overline{u}_xdx+\Im\int_{\R}\phi u\overline{u}_{xt}dx
				=\Im\int_{\R}\phi u_t\overline{u}_xdx-\Im\int_{\R}(\phi u)_x\overline{u}_tdx\\
				=&-\Im\int_{\R}i\phi (iu_t)\overline{u}_xdx-\Im\int_{\R}i(\phi u)_x\overline{iu_{t}}dx
				=-\Re\int_{\R}\phi u_x\overline{iu_t}dx-\Re\int_{\R}(\phi u)_x\overline{iu_t}dx\\
			=&-2\Re\int_{\R}\phi u_x\overline{iu_t}dx-\Re\int_{\R}\phi_x u\overline{iu_t}dx
				=-2\Re\int_{\R}\phi iu_t\overline{u_x}dx-\Re\int_{\R}\phi_x iu_t\overline{u}dx.
			\end{split}
		\end{equation*}
By \eqref{eq7},
		\begin{equation*}
			\begin{split}
				\frac{d}{dt}I(u(t))=&2\Re\int_{\R}\phi u_{xx}\overline{u_x}dx+\Re\int_{\R}\phi_x u_{xx}\overline{u}dx\\
				&-2\Re\int_{\R}\phi  K(x)|u|^{2\sigma}u\overline{u_x}dx-\Re\int_{\R} \phi_xK(x)|u|^{2\sigma}u\overline{u}dx.
			\end{split}
		\end{equation*}
Notice that $2\Re(u_x\overline{u})=2\Re(u\overline{u_x})=(|u|^2)_x$, so
		\begin{equation*}
			\begin{split}
				\frac{d}{dt}I(u(t))=&\int_{\R}\phi (|u_x|^2)_xdx+\Re\int_{\R}\phi_x u_{xx}\overline{u}dx\\
				&-\frac{2}{2\sigma+2}\Re\int_{\R}\phi K(x)(|u|^{2\sigma+2})_xdx-\int_{\R}\phi_xK(x)|u|^{2\sigma+2}dx.
			\end{split}
		\end{equation*}
Integrating by parts, we obtain
\begin{equation}\label{Iu}
			\begin{split}
				\frac{d}{dt}I(u(t))=&-2\int_{\R}\phi_x |u_x|^2dx-\Re\int_{\R}\phi_{xx} u_{x}\overline{u}dx\\
				&+\frac{2}{2\sigma+2}\int_{\R}(\phi K(x))_x |u|^{2\sigma+2}dx-\int_{\R}\phi_xK(x)|u|^{2\sigma+2}dx\\
				=&-2\int_{\R}\phi_x |u_x|^2dx-\frac{1}{2}\int_{\R}\phi_{xx} (|u|^2)_{x}dx\\
				&+\left(\frac{2}{2\sigma+2}-1\right)\int_{\R}\phi_x K(x)|u|^{2\sigma+2}dx+\frac{2}{2\sigma+2}\int_{\R}\phi K^{\prime}(x)|u|^{2\sigma+2}dx.\\
			\end{split}
\end{equation}
Integrating by parts again for the second term in \eqref{Iu}, we obtain \eqref{eq9}.

Applying H\"{o}lder's inequality and \eqref{eq8}, it yields
		\begin{equation}\label{eqbdd}
			|I(u(t))|\leq\|\phi\|_{L^\infty}\|u(t)\|_{L^2}\|u_x(t)\|_{L^2}\leq R\|u(t)\|_{L^2}\|u_x(t)\|_{L^2},
		\end{equation}
which shows the boundedness with respect to $t$.
\end{proof}
 Next, we will conclude the localized  $\dot{H}^1$-norm estimate from \eqref{eq6}.     \begin{lemma}\label{lem3.2-1}
Suppose $0<b<1$ and $u\in H^1(\R)$. For any $R>1$, we have
\begin{equation}\label{eq localGN}
    \int_{|x|\leq R} \phi_x|u_x|^2dx\geq \frac{C}{\|u\|^{4-2b}_{H^1}}\int_{|x|\leq \frac{R}{2}} |x|^{-b}|u|^{6-2b}dx-\frac{c}{R^2}\|u\|^2_{H^1}.
\end{equation}
\end{lemma}
\begin{proof}
Since the supports of  
$\phi_x$ is within $\{x:|x|\leq R\}$, we have
\begin{equation*}
    \int_{|x|\leq\frac{R}{2}}|x|^{-b}|u|^{6-2b}dx\leq\int_\R|\phi_x|^{6-2b}|x|^{-b}|u|^{6-2b}dx=\int_\R|x|^{-b}|\phi_x u|^{6-2b}dx,
\end{equation*}
and
\begin{equation*}
    \int_\R |(\phi_x u)_x|^2 dx=\int_{|x|\leq R}(\phi_x)^2|u_x|^2dx-\int_{|x|\leq R}\phi_x\phi_{xxx}|u|^2dx.
\end{equation*}
The properties of smooth function $\phi$ imply that
\begin{equation*}
     \int_\R |(\phi_x u)_x|^2 dx\leq c\int_{|x|\leq R}\phi_x|u_x|^2dx+\frac{c}{R^2}\int|u|^2dx.
\end{equation*}
Combining this with \eqref{eq6}, we obtain
    \begin{equation*}
       \int_{|x|\leq\frac{R}{2}}|x|^{-b}|u|^{6-2b}dx\leq C_{GN}\|u\|^{4-2b}_{L^2}( c\int_{|x|\leq R}\phi_x|u_x|^2dx+\frac{c}{R^2}\|u\|^2_{H^1}),
    \end{equation*}
from which we can conclude \eqref{eq localGN}.
\end{proof}

Now let us turn to the Morawetz-type estimates.
\begin{prop}\label{lem3.5}
Let $u(t)\in H^1$ be a global solution of \eqref{eq1}. Under assumptions of Theorem \ref{theorem{1}}, the Morawetz-type estimates \eqref{eq3}, \eqref{eq1.5} and \eqref{eq1.2.3} are valid.
\end{prop}
\begin{proof}
Using the properties of function $\phi$, we rewrite \eqref{eq9} as
  \begin{equation*}
      \begin{split}
				-\frac{d}{dt}I(u(t))=&2\int_{|x|\leq R}\phi_x|u_x|^2dx-\frac{1}{2}\int_{\frac{R}{2}<|x|\leq R}\phi_{xxx}|u|^2dx+\frac{2\sigma}{2\sigma+2}\int_{|x|\leq\frac{R}{2}}K(x)|u|^{2\sigma+2}dx\\ &-\frac{2}{2\sigma+2}\int_{|x|\leq\frac{R}{2}}xK^{\prime}(x)|u|^{2\sigma+2}dx+\frac{2\sigma}{2\sigma+2}\int_{\frac{R}{2}<|x|\leq R}\phi_x K(x)|u|^{2\sigma+2}dx\\
    &-\frac{2}{2\sigma+2}\int_{\frac{R}{2}<|x|\leq R}\phi K^{\prime}(x)|u|^{2\sigma+2}dx-\frac{2}{2\sigma+2}\int_{R<|x|}\phi K^{\prime}(x)|u|^{2\sigma+2}dx.
			\end{split}
  \end{equation*}

Firstly, if $K(x)$ satisfies $(K_1)$, all the terms containing $K(x)$ and $K^{\prime}(x)$ are positive, and combining with the second term controlled by $R^{-2}\|u\|^2_{H^1}$, we have
\begin{equation*}
\begin{split}
      \int_\R\frac{|xK^{\prime}(x)|}{1+|x|}|u|^{2\sigma+2}dx&\lesssim\int_{|x|
      \leq\frac{R}{2}}|xK^{\prime}(x)||u|^{2\sigma+2}dx+\int_{\frac{R}{2}<|x|}|\phi K^{\prime}(x)||u|^{2\sigma+2}dx\\
      &\lesssim-\frac{d}{dt}I(u(t))+\frac{1}{R^2},~~~~~~\text{for any}~R>1.
\end{split}
\end{equation*}
This implies \eqref{eq3}.

Next, if $K(x)$ satisfies $(K_2)$ and $2-b<\sigma<\infty$, we have
  \begin{equation*}
       \begin{split}
				-\frac{d}{dt}I(u(t))\geq&2\int_{|x|\leq R}\phi_x|u_x|^2dx-\frac{1}{2}\int_{\frac{R}{2}<|x|\leq R}\phi_{xxx}|u|^2dx-c\int_{|x|\leq\frac{R}{2}}|x|^{-b}|u|^{2\sigma+2}dx\\ &-c\int_{\frac{R}{2}<|x|\leq R}\phi_x |x|^{-b}|u|^{2\sigma+2}dx
    -c\int_{\frac{R}{2}<|x|\leq R}|\phi| |x|^{-b-1}|u|^{2\sigma+2}dx\\
    &-c\int_{R<|x|}R |x|^{-b-1}|u|^{2\sigma+2}dx,
			\end{split}
  \end{equation*}
where $c$ is independent of $R$. Notice that $|\phi^{(\alpha)}(x)|\lesssim R|x|^{-\alpha}$, so
  \begin{equation*}
      -\frac{d}{dt}I(u(t))\geq 2\int_{|x|\leq R}\phi_x|u_x|^2dx-c\int_{|x|\leq\frac{R}{2}}|x|^{-b}|u|^{2\sigma+2}dx-\frac{c}{R^2}\|u\|^2_{H^1}
      -\frac{c}{R^{b}}\|u\|^{2\sigma+2}_{H^1}.
  \end{equation*}
Applying H\"{o}lder's inequality and \eqref{eq localGN}, we get
  \begin{equation*}
      -\frac{d}{dt}I(u(t))\geq\left(\frac{2C}{\|u\|^{2\sigma-4+2b}_{H^1}}-c\right)\int_{|x|\leq\frac{R}{2}}|x|^{-b}|u|^{2\sigma+2}dx
      -\frac{c}{R^2}\|u\|^2_{H^1}-\frac{c}{R^{b}}\|u\|^{2\sigma+2}_{H^1}.
  \end{equation*}
According to Theorem 6.1.4 in \cite{Cz2003}, we know that,  for every $0<\delta<1$, there exists $\epsilon(\delta)>0$ such that, $\|u(0)\|_{H^1(\R)}\leq\epsilon(\delta)$ implies $\sup_{t\in\R}\|u\|_{H^1(\R)}<\delta$.
Thus, we can choose $\delta>0$ small enough so that $\frac{2C}{\delta^{2\sigma-4+2b}}>c+1$. This yields
  \begin{equation*}
      \int_{|x|\leq\frac{R}{2}}|x|^{-b}|u|^{2\sigma+2}dx\lesssim -\frac{d}{dt}I(u(t))+\frac{1}{R^2}+\frac{1}{R^b}.
  \end{equation*}
Integrating over $(0,T)$ and using \eqref{eqbdd}, we have
  \begin{equation*}
      \frac{1}{T}\int_{0}^{T}\int_{|x|\leq\frac{R}{2}}|x|^{-b}|u|^{2\sigma+2}dxdt\lesssim\frac{R}{T}+\frac{1}{R^b},
  \end{equation*}
for $R>1$.

Finally, if $K(x)$ satisfies $(K_3)$ and $0<\sigma\leq 2-b$, we have
\begin{equation*}
       \begin{split}
				-\frac{d}{dt}I(u(t))\geq&2\int_{|x|\leq R}\phi_x|u_x|^2dx-c\int_{|x|\leq\frac{R}{2}}|x|^{-b}(1+|x|)^{-2+b}|u|^{2\sigma+2}dx\\
    &-\frac{1}{2}\int_{\frac{R}{2}<|x|\leq R}\phi_{xxx}|u|^2dx-c\int_{\frac{R}{2}<|x|\leq R}\phi_x |x|^{-b}(1+|x|)^{-2+b}|u|^{2\sigma+2}dx\\
    &-c\int_{\frac{R}{2}<|x|\leq R}|\phi| |x|^{-b-1}(1+|x|)^{-2+b}|u|^{2\sigma+2}dx\\
    &-c\int_{R<|x|}R |x|^{-b-1}(1+|x|)^{-2+b}|u|^{2\sigma+2}dx.
			\end{split}
  \end{equation*}
Using $|\phi^{(\alpha)}(x)|\lesssim R|x|^{-\alpha}$ again, we get
      \begin{equation*}
          -\frac{d}{dt}I(u(t))\geq2\int_{|x|\leq R}\phi_x|u_x|^2dx-c\int_{|x|\leq\frac{R}{2}}|x|^{-b}(1+|x|)^{-2+b}|u|^{2\sigma+2}dx
          -\frac{c}{R^2}(\|u\|^2_{H^1}+\|u\|^{2\sigma+2}_{H^1}).
      \end{equation*}
Considering the second term, we have
  \begin{equation*}
      \int_{|x|\leq\frac{R}{2}}|x|^{-b}(1+|x|)^{-2+b}|u|^{2\sigma+2}dx
      \leq\|u\|^{2\sigma}_{L^\infty}\int_{|x|\leq\frac{R}{2}}|x|^{-b}(1+|x|)^{-2+b}|u|^2dx,
  \end{equation*}
and then, for large $R$, it yields
  \begin{equation}\label{eq3.3}
      \int_{|x|\leq\frac{R}{2}}|x|^{-b}(1+|x|)^{-2+b}|u|^{2\sigma+2}dx\lesssim\|u\|^{2\sigma}_{L^\infty}\int_{|x|\leq1}|x|^{-b}|u|^2dx
      +\|u\|^{2\sigma}_{L^\infty}\int_{1<|x|\leq\frac{R}{2}}(1+|x|)^{-2}|u|^2dx.
  \end{equation}
To deal with the first term on the right-hand side of \eqref{eq3.3}, applying Poincare's inequality and the argument of cutoff like Lemma \ref{lem3.2-1} (or by Exercises 8.5 in \cite{Brezis}), it gives
  \begin{equation*}
      \int_{|x|\leq1}|x|^{-b}|u|^2dx\leq c\|u\|^{2}_{L^\infty_{(|x|\leq1)}}\leq c\|u_x\|^2_{L^2_{(|x|\leq2)}}.
  \end{equation*}
For the second term on the right-hand side of \eqref{eq3.3}, we apply the following inequality for $u\in H^1(\R)$ (refer to Corollary 11 in \cite{FLW2022})
  \begin{equation*}
      \int_{1<|x|<\infty}|x|^{-2}|u|^2dx\leq c(\int_{1<|x|<\infty}|u_x|^2dx+\|u\|^{2}_{L^\infty_{(|x|\leq1)}})\leq c\int_{\R}|u_x|^2dx.
  \end{equation*}
By the same arguments as in Lemma \ref{lem3.2-1}, we have
  \begin{equation*}
      c\int_{1<|x|\leq\frac{R}{2}}(1+|x|)^{-2}|u|^2dx-\frac{c}{R^2}\|u\|^2_{H^1}\leq \int_{|x|\leq R}\phi_x|u_x|^2dx,
  \end{equation*}
and
  \begin{equation*}
      \int_{|x|\leq R} \phi_x|u_x|^2dx\geq \frac{c}{\|u\|^{2\sigma}_{H^1}}\int_{|x|\leq \frac{R}{2}} |x|^{-b}(1+|x|)^{-2+b}|u|^{2\sigma+2}dx-\frac{c}{R^2}\|u\|^2_{H^1}.
  \end{equation*}
Similarly, by selecting $\delta$ and $\epsilon(\delta)$ small enough,  we find that for any $T>0$ and sufficiently large $R>0$, the following holds
  \begin{equation*}
      \frac{1}{T}\int_{0}^{T}\int_{|x|\leq R}|x|^{-b}(1+|x|)^{-2+b}|u|^{2\sigma+2}dxdt\lesssim\frac{R}{T}+\frac{1}{R^2}.
      \end{equation*}
  \end{proof} 

Now it is ready to prove Theorem \ref{theorem{1}}.
	
\begin{proof}[Proof of Theorem \ref{theorem{1}}]
Firstly, we confirm that
there exists a sequence $t_n\rightarrow\infty$ such that $\|u(t_n)\|_{L^2(I)}\rightarrow0$.

If $K$ satisfies $(K_1)$, from the proof of Proposition \ref{lem3.5} and by applying Poincare inequality, we obtain
\begin{equation*}
    \int_{|x|\leq \frac{R}{2}}|u|^2 dx\leq R^2\int_{|x|\leq R^2}\phi_x |u_x|^2 dx\lesssim -R^2\frac{d}{dt}I(u)+\frac{R^2}{R^4}.
\end{equation*}
Integrating over time and taking $R=T^{\frac{1}{8}}$, it yields that for some $0<\alpha<1$,
\begin{equation*}
    \int^T_0\int_{|x|\leq \frac{R}{2}}|u(t)|^2 dx\lesssim T^\alpha,
\end{equation*}
which implies that there exists a sequence of times $t_n\rightarrow\infty$ such that, for any $R>0$,
\begin{equation*}
    \lim_{n\rightarrow\infty}\int_{|x|\leq R}|u(t_n)|^2 dx=0.
\end{equation*}
If $K$ satisfies $(K_2)$ or $(K_3)$, we have
\begin{equation*}
    \int_{|x|\leq R}\phi_x |u_x|^2 dx\lesssim -\frac{d}{dt}I(u)+\frac{1}{R^b},
\end{equation*}
or
\begin{equation*}
    \int_{|x|\leq R}\phi_x |u_x|^2 dx\lesssim -\frac{d}{dt}I(u)+\frac{1}{R^2},
\end{equation*}
respectively. Similarly, we choose $R=T^\theta$ such that, for
some $0<\alpha<1$,
\begin{equation*}
    \int^T_0\int_{|x|\leq \frac{R}{2}}|u(t)|^2 dx\lesssim T^\alpha,
\end{equation*}
which implies that there exists a sequence of times $t_n\rightarrow\infty$ such that for any $R>0$,
\begin{equation*}
    \lim_{n\rightarrow\infty}\int_{|x|\leq R}|u(t_n)|^2 dx=0,
\end{equation*}
which implies that, for any bounded interval $I\subset\R$,
\begin{equation}\label{eq18}
			\lim_{t_n\rightarrow\infty}\|u(t_n)\|_{L^2(I)}=0.
\end{equation}

Next, we will show that $\|u(t_n)\|_{L^\infty(I)}\rightarrow0$. We claim that for any interval  $I\subset\R$, there exists a function $\widetilde{x}(t_n)\in I$ such that
\begin{equation}\label{eq19}
			\lim_{n\rightarrow\infty} |u(t_n),\widetilde{x}(t_n)|^2=0.
\end{equation}

To prove \eqref{eq19} by contradiction, consider an interval $I$ and assume that for any positive integer $n$, there exists a subsequence $t_{n_k}>n$ and a constant $\varepsilon_0>0$ such that
\begin{equation}\label{eq19-1}
			|u(t_{n_k},x)|^2>\varepsilon_0,\qquad\forall x\in I.
\end{equation}
By integrating \eqref{eq19-1} over $I$, it gives
		$$\int_I|u(t_{n_k},x)|^2dx>|I|\varepsilon_0,$$
which contradicts \eqref{eq18}. By  H\"{o}lder's inequality, for any $x\in I$,
\begin{equation*}		\left||u(t_n,x)|^2-|u(t_n,\widetilde{x}(t_n))|^2\right|=\left|\int_{\widetilde{x}(t_n)}^{x}(|u|^2)_xdx\right|\leq2\int_I|u||u_x|dx\leq2\|u(t_n)\|_{L^2(I)}\|u_x(t_n)\|_{L^2(I)}.
\end{equation*}
By combining \eqref{eq18} and \eqref{eq19}, and letting $t_n\rightarrow\infty$, we get
		$$|u(t_n,x)|^2\rightarrow0,\qquad\forall x\in I,$$
which concludes \eqref{eq2}. This completes the proof of Theorem \ref{theorem{1}}. 
  \end{proof}

To show the application of the decay estimates \eqref{eq2},
we introduce the concept of pre-compactness used in many previous works (see \cite{DHR2008}).
Let $u(t)\in H^1$, we call $u(t)$ is pre-compact in $H^1$, if for any
time sequence $\{t_n\}\subset\R$, there exists a
function $\tilde{u}\in H^1$  depending on $\{t_n\}$ such that, up to a subsequence, $u(t_n)\rightarrow\tilde{u}$ in $H^1$. The pre-compactness implies uniform localization; we have

\begin{prop}\label{prop 5.1}
Let $u$ be a solution of \eqref{eq7} and $u(t)$  is pre-compact in $H^1(\R)$. Then for each $\varepsilon>0$, there exists a constant $R(\varepsilon)>0$ such that
    \begin{equation*}
        \int_{|x|>R(\varepsilon)} |u(t,x)|^2dx\leq\varepsilon, \qquad  \forall t\in\R.
    \end{equation*}
\end{prop}

\begin{proof}
Let us prove the proposition by contradiction.  Assume the conclusion is false,  then there exists a constant $A>0$ and a sequence of time $t_n$ such that
        \begin{equation*}
            \int_{|x|>R}|u(t_n,x)|^2dx\geq A
        \end{equation*}
for all $R>0$. Since $u(t)$ is pre-compact, there exists $\tilde{u}\in H^1$ such that $u(t_n)\rightarrow\tilde{u}$ in $H^1$, by passing to a subsequence of $t_n$, we can affirm that for sufficiently large $R>0$,
\begin{equation*}
            \int_{|x|>R}|\tilde{u}(x)|^2dx>\frac{A}{2},
\end{equation*}
which contradicts the fact that $\tilde{u}\in H^1$.
\end{proof}

\begin{thm}
   Under the assumptions of Theorem \ref{theorem{1}}, if $u(t)$ is the global solution of \eqref{eq7} with initial data $u_0\neq0$, then $u(t)$ is not pre-compact.
\end{thm}
\begin{proof} Let us prove the theorem by contradiction. If the conclusion is false, then there exists $u_0\neq0$ such that the corresponding solution $u(t)$ is pre-compact in $H^1(\R)$.  Applying the mass conservation of \eqref{eq7}, we have
    \begin{equation*}
        \int_\R |u_0|^2dx=\int_\R |u(t)|^2dx=\int_{|x|\leq R} |u(t)|^2dx+\int_{|x|> R} |u(t)|^2dx
    \end{equation*}
    for all $t\in\R$ and $R>0$. Combining Proposition \ref{prop 5.1} and Theorem \ref{theorem{1}},  taking $t=t_n$ as the sequence of times in \ref{theorem{1}}, we obtain
    \begin{equation*}
        \int_\R |u_0|^2dx\leq\lim_{n\rightarrow\infty}\int_{|x|\leq R(\varepsilon)} |u(t_n)|^2dx+\varepsilon\leq2\varepsilon.
    \end{equation*}
Due to the arbitrariness of $\varepsilon$ and mass conservation, this leads to a contradiction with the assumption that $u_0\neq0$.
    \end{proof}

Theorem 3.5 implies that, under the assumptions of Theorem \ref{theorem{1}}, if $u(t)$ is global solution of \eqref{eq7} and $u(t)$ is pre-compact, then $u(t)\equiv 0$.
Using the concentration-compactness method developed by Kenig-Merle \cite{KM2006}, for the INLS with $N\geq2$ and for the classical NLS, the pre-compactness of non-scattering solutions has been extensively discussed \cite{CFGM2020,CHL2020,Dinh2021-3,FG2017,FG2019,GM2021,KM2006,MMZ2021,HR2008,DHR2008,FXC2011}. To obtain a scattering solution, one can exclude the existence of non-scattering solutions under appropriate conditions by proving that such solutions are pre-compact.

\subsection{Proof of Theorem \ref{T1}}\label{s3.2}

Follows the ideas in \cite{M2020No}, it is achieved by choosing a weight function $\phi\in C^1(\R)\cap C^2(\R\setminus \{0\})$ and combining it with the oddness condition and the application of \eqref{eq13}. In this subsection, we start with a Virial identity associated with \eqref{eq7}, assume that $u(t)\in H^1_{odd}(\R)$ is a solution of equation \eqref{eq7} and set $\phi(x)=\frac{x}{1+|x|}$. Define
\begin{equation}\label{eq8.2}
		I(u(t)):=\Im\int_{\R}\phi(x)u(t,x)\overline{u}_x(t,x)dx,
\end{equation}

\begin{lemma}\label{lem3.2.1}
Let $u\in C(\R;H^1_{odd}(\R))$ be a solution of equation \eqref{eq7}. Then $I(u(t))$ is well-defined and bounded in time. Moreover, the following Virial identity holds:
\begin{equation}\label{eq9.2}
\begin{split}
				-\frac{d}{dt}I(u(t))=&2\int_{\R}\frac{1}{(1+|x|)^2}|u_x|^2dx-\frac{1}{2}\int_{\R}\frac{6}{(1+|x|)^4}|u|^2dx\\ &-\left(\frac{2}{2\sigma+2}-1\right)\int_{\R}\frac{K(x)}{(1+|x|)^2}|u|^{2\sigma+2}dx-\frac{2 }{2\sigma+2}\int_{\R}\frac{xK^{\prime}(x)}{(1+|x|)}|u|^{2\sigma+2}dx.
\end{split}
\end{equation}
\end{lemma}
\begin{proof}
Let $u(t)\in H^1_{odd}(\R)$ be a solution of equation \eqref{eq7}. Similar as the proof of Lemma \ref{lem3.1}, we obtain
\begin{equation}\label{Iu2}
\begin{split}
				\frac{d}{dt}I(u(t))=&-2\int_{\R}\phi_x |u_x|^2dx-\Re\int_{\R}\phi_{xx} u_{x}\overline{u}dx\\
				&+\frac{2}{2\sigma+2}\int_{\R}(\phi K(x))_x |u|^{2\sigma+2}dx-\int_{\R}\phi_xK(x)|u|^{2\sigma+2}dx\\
				=&-2\int_{\R}\phi_x |u_x|^2dx-\frac{1}{2}\int_{\R}\phi_{xx} (|u|^2)_{x}dx\\
				&+\left(\frac{2}{2\sigma+2}-1\right)\int_{\R}\phi_x K(x)|u|^{2\sigma+2}dx+\frac{2}{2\sigma+2}\int_{\R}\phi K^{\prime}(x)|u|^{2\sigma+2}dx.\\
\end{split}
\end{equation}
Taking $\phi=\frac{x}{1+|x|}$, then $\phi_x=\frac{1}{(1+|x|)^2}$, and when $x\neq0,~\phi_{xx}=-\frac{2x}{(1+|x|)^3|x|}$, $\phi_{xxx}=\frac{6}{(1+|x|)^4}$. Considering the second term in \eqref{Iu2}, we have
\begin{equation*}
			\int_{\R}\phi_{xx} (|u|^2)_{x}dx=\int_{-\infty}^{0}\phi_{xx} (|u|^2)_{x}dx+\int_{0}^{+\infty}\phi_{xx} (|u|^2)_{x}dx.
\end{equation*}
Integrating by parts, it yields
\begin{equation*}
			\int_{\R}\phi_{xx} (|u|^2)_{x}dx=-\left(\int_{-\infty}^{0}	\frac{6}{(1+|x|)^4} |u|^2dx+\int_{0}^{+\infty}	\frac{6}{(1+|x|)^4} |u|^2dx\right)+4|u(t,0)|^2.
\end{equation*}
Combining with \eqref{Iu2}, we get
\begin{equation*}
\begin{split}
				\frac{d}{dt}I(u(t))=&-2\int_{\R}\frac{1}{(1+|x|)^2}|u_x|^2dx+\frac{1}{2}\int_{\R}\frac{6}{(1+|x|)^4}|u|^2dx\\ &+\left(\frac{2}{2\sigma+2}-1\right)\int_{\R}\frac{K(x)}{(1+|x|)^2}|u|^{2\sigma+2}dx+\frac{2 }{2\sigma+2}\int_{\R}\frac{xK^{\prime}(x)}{(1+|x|)}|u|^{2\sigma+2}dx.
\end{split}
\end{equation*}
Applying H\"{o}lder's inequality and \eqref{eq8.2},
\begin{equation}\label{eqbdd2}
			|I(u(t))|\leq\|\phi\|_{L^\infty}\|u(t)\|_{L^2}\|u_x(t)\|_{L^2}\leq\|u(t)\|_{L^2}\|u_x(t)\|_{L^2},
\end{equation}
which shows the boundedness with respect to $t$.
\end{proof}

In the next, we will give some estimates to $-\frac{d}{dt}I(u(t))$ in the context of the $H^1_\alpha$-norm for certain weighted function $\alpha(x)$.

\begin{lemma}\label{lem3.2.2}
If $u$ is an odd solution of \eqref{eq7} as stated in Theorem \ref{T1}, then
\begin{equation}\label{eq14}
			-\frac{d}{dt}I(u(t))\geq C\|u(t)\|^2_{H^1_\alpha(\R)},
\end{equation}
where $C>0,~\alpha(x)=\frac{1}{(1+|x|)^4}$.
\end{lemma}
\begin{proof}
By \eqref{eq9.2},
\begin{equation*}
\begin{split}
				-\frac{d}{dt}I(u(t))=&B(u_1)+B(u_2)\\
				&-\left(\frac{2}{2\sigma+2}-1\right)\int_{\R}\frac{K(x)}{(1+|x|)^2}|u|^{2\sigma+2}dx-\frac{2}
{2\sigma+2}\int_{\R}\frac{xK^{\prime}(x)}{1+|x|}|u|^{2\sigma+2}dx\\
				=&I_1+I_2+I_3,
\end{split}	
\end{equation*}
where
\begin{equation*}
\begin{split}
I_1&=B(u_1)+B(u_2),\\
I_2&=-\left(\frac{2}{2\sigma+2}-1\right)\int_{\R}\frac{K(x)}{(1+|x|)^2}|u|^{2\sigma+2}dx,\\
I_3&=-\frac{2 }{2\sigma+2}\int_{\R}\frac{xK^{\prime}(x)}{1+|x|}|u|^{2\sigma+2}dx.
\end{split}
\end{equation*}
To complete the proof, we need only to estimate the terms $I_2$ and $I_3$, as the term involving  $B(u_k)$ has already been discussed in Lemma \ref{lem3.2}.
		
If $K(x)$ satisfies $(K_1)$, then $\frac{2}{2\sigma+2}<1$ and $xK^{\prime}(x)<0$ imply that $I_2$, $I_3\geq0$, so we can get \eqref{eq14} from Lemma \ref{lem3.2}.
		
If $K(x)$ satisfies $(K_4)$, we only need to estimate $\int_{\R}\frac{1}{(1+|x|)^4}|x|^{-b}|u|^{2\sigma+2}dx$, since we have $|I_2|$, $|I_3|\lesssim\int_{\R}\frac{1}{(1+|x|)^4}|x|^{-b}|u|^{2\sigma+2}dx$.
Applying the inequality
\begin{equation*}			\int_{\R}\frac{1}{(1+|x|)^4}|x|^{-b}|u|^{2\sigma+2}dx\leq\|u\|_{L^\infty}^{2\sigma}\int_{\R}|x|^{-b}\left|\frac{1}{(1+|x|)^2}u\right|^{2}dx
\end{equation*}
and \eqref{eq6.1}, it gives
\begin{equation*}
			\int_{\R}\frac{1}{(1+|x|)^4}|x|^{-b}|u|^{2\sigma+2}dx\leq C\|u\|_{L^\infty}^{2\sigma}\left(\left\| (\frac{1}{(1+|x|)^2}u)_x\right\|_{L^2(\R)}^2+\left\|\frac{1}{(1+|x|)^2}u\right\|_{L^2(\R)}^2\right).
\end{equation*}
By the similar arguments as in Lemma \ref{lem3.2}, we get
		\begin{equation*}
			\int_{\R}\left[\left(\frac{1}{(1+|x|)^2}u\right)_x\right]^2\leq	\int_{\R}\frac{1}{(1+|x|)^4}|u_x|^{2}dx,
		\end{equation*}
and then
		\begin{equation}\label{eqM}
			\int_{\R}\frac{1}{(1+|x|)^4}|x|^{-b}|u|^{2\sigma+2}dx\leq C\|u\|_{H^1(\R)}^{2\sigma}\|u(t)\|^2_{H^1_\alpha(\R)}.
		\end{equation}
Similarly applying Theorem 6.1.4 of \cite{Cz2003} again,choosing $\delta>0$ small enough, there exists $\epsilon(\delta)>0$ such that $\|u(0)\|_{H^1(\R)}\leq\epsilon(\delta)$ implies that $\sup_{t\in\R}\|u\|_{H^1(\R)}<\delta$ to get \eqref{eq14}.
This completes the proof.
\end{proof}
\begin{cor} If $K(x)$ satisfies $(K_1)$, we have
		\begin{equation}\label{eq15}
			\int_{\R}\frac{|xK^{\prime}(x)|}{1+|x|}|u|^{2\sigma+2}dx\lesssim -\frac{d}{dt}I(u(t)).
		\end{equation}If $K(x)$ satisfies $(K_4)$ and $\|u_0\|_{H^1(\R)}\leq\epsilon$, where $\epsilon$ is small enough, then we have
		\begin{equation}\label{eq16}
			\int_{\R}\frac{1}{(1+|x|)^4}|x|^{-b}|u|^{2\sigma+2}dx\lesssim-\frac{d}{dt}I(u(t)).
		\end{equation}
\end{cor}
\begin{proof}
If $K(x)$ satisfies $(K_1)$, then  by the proof of Lemma \ref{lem3.2}, one has $I_3\lesssim -\frac{d}{dt}I(u(t))$, which implies \eqref{eq15}.

If $K(x)$ satisfies $(K_4)$, as stated in the end of the proof of Lemma \ref{lem3.2.2}, from $\|u_0\|_{H^1(\R)}\leq\epsilon$, by \eqref{eqM} and \eqref{eq14}, we can get \eqref{eq16}.
\end{proof}

\begin{lemma}\label{lem3.5.2}
		There exists a constant $M>0$ such that
\begin{equation}\label{eq17}
			\int_{0}^{\infty} \|u(t)\|^2_{H^1_\alpha(\R)}dt\leq M.
\end{equation}
\end{lemma}
\begin{proof}
For any $\tau>0$, integrating \eqref{eq14} over $[0,~\tau]$, we get
		$$\int_0^\tau \|u(t)\|^2_{H^1_\alpha(\R)}dt\leq C(I(u(0))-I(u(\tau))).$$
It follows from \eqref{eqbdd2} that
$$\int_0^\tau \|u(t)\|^2_{H^1_\alpha(\R)}dt\leq C\sup_{t\in[0,~\tau]}\|u(t)\|^2_{H^1(\R)},$$
taking $\tau\rightarrow\infty$, we get the desired result.
\end{proof}
Now we can prove Theorem \ref{T1}.
\begin{proof}[Proof of Theorem \ref{T1}]
Firstly, we confirm that  $\|u\|_{L^2(I)}\rightarrow0$. Notice that
		\begin{equation*}
			\frac{d}{dt}\left(\frac{1}{2}\int_{\R}\psi|u(t)|^2\right)=\Re\int_{\R}\psi \overline{u}u_tdx=\Im\int_{\R}\psi \overline{u}(iu_t)dx.
		\end{equation*}
Let $\psi(x)=\frac{1}{(1+|x|)^4}$. Using \eqref{eq7} and integrating by parts,
		\begin{equation*}
			\begin{split}
				\frac{d}{dt}\left(\frac{1}{2}\int_{\R}\psi|u(t)|^2\right)&=-\Im\int_{\R}\psi\overline{u}u_{xx}+\Im\int_{\R}\psi K(x)|u|^{2\sigma}u\overline{u}dx\\
				&=\Im\left(\int_{\R}\psi|u_x|^2dx+\int_{\R}\psi_x\overline{u}u_xdx+\int_{\R}\psi K(x)|u|^{2\sigma+2}dx\right)\\
				&=\Im\int_{\R}\psi_x\overline{u}u_xdx.
			\end{split}
		\end{equation*}
As $|\psi_x(x)|=\frac{4}{(1+|x|)^5}\lesssim\frac{1}{(1+|x|)^4}$, applying H\"{o}lder's inequality, we have
\begin{equation}\label{psi}		\left|\frac{d}{dt}\left(\frac{1}{2}\int_{\R}\psi|u(t)|^2\right)\right|\lesssim\int_{\R}\frac{1}{(1+|x|)^4}\left(|u(t,x)|^2+|u_x(t,x)|^2\right)dx=\|u(t)\|^2_{H^1_\alpha(\R)},
		\end{equation}
where $\alpha(x)=\frac{1}{(1+|x|)^4}$. It follows from \eqref{eq17} that there exists a sequence $t_n\in\R$ such that $\|u(t_n)\|^2_{H^1_\alpha}\rightarrow0$, which implies that $\|u(t_n)\|^2_{L^2_\alpha}\rightarrow0$. For any $t\in\R$, integrating \eqref{psi} over $[t,t_n]$, and let $t_n\rightarrow\infty$, we get
\begin{equation*}
			\|u(t)\|^2_{L^2_\alpha(\R)}\lesssim\int_{t}^{\infty}\|u(s)\|^2_{H^1_\alpha(\R)}ds.
\end{equation*}
In consequence, $\lim_{t\rightarrow\infty}\int_{\R}\frac{1}{(1+|x|)^4}|u(t)|^2dx=0$, which implies that
		\begin{equation}\label{eq18.1}
			\lim_{t\rightarrow\infty}\|u(t)\|_{L^2(I)}=0
		\end{equation}
holds for any bounded interval $I\subset\R$.

Now let us show that $\|u\|_{L^\infty(I)}\rightarrow0$. We claim that for any interval  $I\subset\R$, there exists $\widetilde{x}(t)\in I$ such that
		\begin{equation}\label{eq19.1}
			\lim_{t\rightarrow\infty} |u(t,\widetilde{x}(t))|^2=0.
		\end{equation}
We prove \eqref{eq19.1} by contradiction. Let $I$ be an interval and suppose that for any positive integer $n$, there is a $t_n>n$ and an $\varepsilon_0>0$ such that
\begin{equation}\label{eq19-2}
			|u(t_n,x)|^2>\varepsilon_0,\qquad\forall x\in I.
\end{equation}
By integrating \eqref{eq19-2} over $I$, it gives
		$$\int_I|u(t_n,x)|^2dx>|I|\varepsilon_0,$$
which contradicts \eqref{eq18.1}. By  H\"{o}lder's inequality, for any $x\in I$,
\begin{equation*}	||u(t,x)|^2-|u(t,\widetilde{x}(t))|^2|=\left|\int_{\widetilde{x}(t)}^{x}(|u|^2)_xdx\right|\leq2\int_I|u||u_x|dx\leq2\|u(t)\|_{L^2(I)}\|u_x(t)\|_{L^2(I)}.
\end{equation*}
Combining \eqref{eq18.1} and \eqref{eq19.1}, let $t\rightarrow\infty$, we get
		$$|u(t,x)|^2\rightarrow0,\qquad\forall x\in I,$$
which implies \eqref{eq decay}.

Finally, integrating \eqref{eq15} and \eqref{eq16} over $[-t,t]$, respectively, it yields $$\int_{-t}^{t}\int_{\R}\frac{|xK^{\prime}(x)|}{1+|x|}|u|^{2\sigma+2}dxds\lesssim(|I(u(-t))|+|I(u(t))|)\lesssim\sup\|u\|^2_{H^1(\R)},$$
and
$$\int_{-t}^{t}\int_{\R}\frac{1}{(1+|x|)^4|x|^{b}}|u|^{2\sigma+2}dxds\lesssim(|I(u(-t))|+|I(u(t))|)\lesssim\sup\|u\|^2_{H^1(\R)}.$$
		Let $t\rightarrow\infty$, we get \eqref{eq Mod} and \eqref{eq Mof}. This completes the proof of Theorem \ref{T1}.
\end{proof}

\section{INLS with potential}

In this section, we consider the equation
	\begin{equation}\label{eq20}
		iu_t+u_{xx}= K(x)|u|^{2\sigma}u+\mu V(x)u,\qquad u\in H^1_{odd}(\R).
	\end{equation}
for the case $\mu\neq0$. Our aim is to prove Theorem \ref{theorem{2}}. 
Inspired by the ideas in \cite{M2020No},
we establish a Virial identity involving external potentials, and give several lemmas.
Set
	$$I(u(t))=\Im\int_\R \phi(x)u(t,x)\overline{u}_x(t,x)dx.$$

Firstly, let us prove $\eqref{eqv1}$ under condition $(1)$ in Theorem \ref{theorem{2}}. In this case $\phi(x)$ is taken the same as in Sect. \ref{s3.1}.
\begin{lemma}\label{lem 4.1}
Let $u\in C(\R;H^1_{odd}(\R))$ be a global solution of equation \eqref{eq20}, then
  \begin{equation}\label{eq21}
  \begin{split}
				-\frac{d}{dt}I(u(t))=&2\int_{\R}\phi_x|u_x|^2dx-\frac{1}{2}\int_{\R}\phi_{xxx}|u|^2dx-\mu\int_\R\phi V^{\prime}(x)|u|^2dx.\\ &-\left(\frac{2}{2\sigma+2}-1\right)\int_{\R}\phi_x K(x)|u|^{2\sigma+2}dx-\frac{2}{2\sigma+2}\int_{\R}\phi K^{\prime}(x)|u|^{2\sigma+2}dx.
 \end{split}
		\end{equation}

\end{lemma}
\begin{proof}[Sketch of proof]
Similar to the proof of Lemma \ref{lem3.1}, by \eqref{eq20} we have
		\begin{equation*}
			\begin{split}
				\frac{d}{dt}I(u(t))=&2\Re\int_{\R}\phi u_{xx}\overline{u_x}dx+\Re\int_{\R}\phi_x u_{xx}\overline{u}dx-2\mu\Re\int_\R\phi Vu\overline{u}_xdx-\mu\Re\int_\R\phi_xVu\overline{u}dx\\
				&-2\Re\int_{\R}\phi  K(x)|u|^{2\sigma}u\overline{u_x}dx-\Re\int_{\R}\phi_xK(x)|u|^{2\sigma}u\overline{u}dx.
			\end{split}
		\end{equation*}
We only need to estimate the terms involving the potential $V$. Using integrating by parts, we get
\begin{equation*}
			2\Re\int_\R\phi Vu\overline{u}_xdx+Re\int_\R\phi_xVu\overline{u}dx=-\int_\R\phi V^{\prime}(x)|u|^2dx.
\end{equation*}
We conclude \eqref{eq21} by the same argument as in Lemma \ref{lem3.1}.
	\end{proof}

Similar to the arguments in Sect. \ref{s3.1}, to deal with the term involving the potential, we divided $\mathbb{R}$ into $\{|x|\leq \frac{R}{2}\}$ and $\{|x|>\frac{R}{2}\}$. Firstly, we use the properties of function $\phi$ to get
\begin{equation*}
    \left|\int_\R \phi V^{\prime}(x)|u|^2dx\right|\leq\int_\R |x| |V^{\prime}(x)||u|^2dx=\int_{|x|\leq\frac{R}{2}} \phi^2_x |x V^{\prime}(x)||u|^2dx+\int_{|x|>\frac{R}{2}}|x V^{\prime}(x)||u|^2dx.
\end{equation*}
To estimate the second term, we have
\begin{equation*}
    \int_{|x|>\frac{R}{2}}|x V^{\prime}(x)||u|^2dx\lesssim\|u\|^2_{L^\infty}\int_{|x|>\frac{R}{2}}|x V^{\prime}(x)|dx\lesssim\sup_{|x|>R}\frac{1}{(1+|x|)^2}\int_{|x|>\frac{R}{2}}|x V^{\prime}(x)|(1+|x|)^2dx.
\end{equation*}
Combining with condition $(V)$, it yields
\begin{equation*}
     \int_{|x|>\frac{R}{2}}|x V^{\prime}(x)||u|^2dx\lesssim\frac{1}{R^2}.
\end{equation*}
The remainder is to estimate the term on $\{|x|\leq \frac{R}{2}\}$. We have
\begin{equation*}
    \int_{|x|\leq\frac{R}{2}} \phi^2_x |x V^{\prime}(x)||u|^2dx\leq\int_\R |x V^{\prime}(x)||\phi_xu|^2dx.
\end{equation*}
Let $V_0=-|xV^{\prime}(x)|$, combining Lemma \ref{lemma2.1} and condition $(V)$, we conclude that there exists a $\mu(V)>0$ such that $-\frac{d^2}{dx^2}+\mu V_0$ has a unique negative eigenvalue for all $\mu<\mu(V)$. Furthermore, the corresponding eigenfunction is even, that is
\begin{equation*}
    \mu\int_\R V_0|u|^2dx+\int_\R|u_x|^2dx\geq0, ~~~~\forall u\in H^1_{odd}(\R).
\end{equation*}
By taking $u=\phi_xu$, we have
\begin{equation*}
    \mu\int_{|x|\leq\frac{R}{2}} \phi^2_x |x V^{\prime}(x)||u|^2dx\leq \int_\R|(\phi_x u)_x|^2dx\lesssim \int_\R\phi_x|u_x|^2dx+\frac{1}{R^2}.
\end{equation*}
Hence, there exits a constant $\mu_0>0$ such that for all $0<\mu<\mu_0$,
\begin{equation*}
    \begin{split}
-\frac{d}{dt}I(u(t))\geq&\frac{3}{2}\int_{\R}\phi_x|u_x|^2dx-\frac{1}{2}\int_{\R}\phi_{xxx}|u|^2dx-\frac{1}{R^2}\\ &-\left(\frac{2}{2\sigma+2}-1\right)\int_{\R}\phi_x K(x)|u|^{2\sigma+2}dx-\frac{2}{2\sigma+2}\int_{\R}\phi K^{\prime}(x)|u|^{2\sigma+2}dx.
			\end{split}
\end{equation*}

Parallel to Sect. \ref{s3.1}, we have
\begin{lemma}\label{lem4.2}
Under the assumptions of Theorem \ref{theorem{2}}, if $u\in C(\R;H^1_{odd}(\R))$ is a global solution
of \eqref{eq20}, then
  		\begin{equation*}
			-\frac{d}{dt}I(u(t))+\frac{1}{R^b}\gtrsim \int_{|x|<R}\phi_x |u_x|^2dx.
		\end{equation*}
\end{lemma}
\begin{proof} [Sketch of proof]
Since we have estimated all terms in \eqref{eq21} previously, here we need only show that $V$ satisfies the assumptions  of Theorem 6.1.4  in \cite{Cz2003}. We confirm that $V\in L^1+L^{\infty}$. Indeed, fixing an $x_0\in[1,\infty)$, then for any $x\in[1,\infty)$, by the hypothesis $(V)$,
		$$|V(x_0)-V(x)|\leq\left|\int_{x_0}^{x}V^{'}(s)ds\right|\leq\int_{1}^\infty |sV^{'}(s)|ds<\infty,$$
thus we have  $V|_{[1,\infty)}\in L^\infty[1,\infty)$. On the other hand, $$\int_{0}^{1}|V|dx\leq\limsup_{\rho\to 0^+}\left(x|V||^1_\rho-\int_{\rho}^{1}x|V|^{'}dx\right)\leq \limsup_{\rho\to 0^+}\left(|V(1)|+\rho|V(\rho)|+\int_{\rho}^{1}x|V^{'}|dx\right).$$We claim that $\limsup_{x\to 0^+}(x|V|)<\infty$. Indeed, set $$V^1(x)=\int^\frac{1}{2}_x|V^{'}(s)|ds+|V(\frac{1}{2})|$$ for $ 0<x<\frac{1}{2}$, then  $V^1(x)$ is a monotonic decreasing function and $$\liminf_{x\to 0^+}\frac{V^1(x)}{|V(x)|}\geq 1.$$ Since $\int_{0}^{1}x|V^{\prime}(x)|dx<\infty$, we deduce that $\limsup_{x\to 0^+} xV^1(x)<\infty$, which implies  $\limsup_{x\to 0^+}(x|V|)<\infty$. Therefore,  $V|_{(0,1)}\in L^1(0,1)$, and we get $V\in L^1+L^\infty$. Finally, taking $\mu$ small enough, we can verify that $g(u)=\mu Vu+ K|u|^{2\sigma}u$ satisfies the the assumptions in Theorem 6.1.4 in \cite{Cz2003}, and the conclusion can be obtained.
\end{proof}
	
\begin{proof}[Proof of \eqref{eqv1}]
With Lemma \ref{lem4.2} established, using similar reasoning as in Sect. \ref{s3.1}, we can derive
 \eqref{eqv1}.
 
\end{proof}

Finally, we provide the proof of \eqref{eqv2} under condition $(2)$ of Theorem \ref{theorem{2}}.
In this case, we take $\phi(x)=\frac{x}{1+|x|}$.
\begin{cor}
Let $u\in C(\R;H^1_{odd}(\R))$ be a global solution of equation \eqref{eq20}, then
\begin{equation}\label{eq21.1}
\begin{split} -\frac{d}{dt}I(t)=&2\int_{\R}\frac{1}{(1+|x|)^2}|u_x|^2dx-\frac{1}{2}\int_{\R}\frac{6}{(1+|x|)^4}|u|^2dx-\mu\int_{\R}\frac{xV^{\prime}(x)}{1+|x|}|u|^2dx\\ &-(\frac{2}{2\sigma+2}-1)\int_{\R}\frac{K(x)}{(1+|x|)^2}|u|^{2\sigma+2}dx-\frac{2 }{2\sigma+2}\int_{\R}\frac{xK^{\prime}(x)}{(1+|x|)}|u|^{2\sigma+2}dx	.	
\end{split}	
\end{equation}
\end{cor}

The proof is similar to that of Lemma \ref{lem 4.1} and Lemma \ref{lem3.2.1}, we omit the details here.

Let $w\in H^1(\R)$ be a real-valued function.  To deal with the term involving the potentials, we introduce a quadratic form as $$B_V(w)=-\frac{1}{2}\int_{\R}\frac{6}{(1+|x|)^4}|w|^2dx-\mu\int_{\R}\frac{xV^{\prime}(x)}{1+|x|}|w|^2dx+2\int_{\R}\frac{1}{(1+|x|)^2}|w_x|^2dx.$$
	
We claim that  for $\alpha(x)=\frac{1}{(1+|x|)^4}$, the following inequality holds for the odd solution $u$:
	\begin{equation*}
		\|u\|^2_{H^1_\alpha(\R)}\lesssim B_V(u_1)+B_V(u_2).
	\end{equation*}
Indeed, let $w=u_k\in H^1_{odd}(\R),~k=1,~2$ and write $B_V(w)=H(w)+J(w)$, where
	\begin{equation*}
		\begin{split}
			H(w)=&-\int_{\R}\frac{3}{(1+|x|)^4}|w|^2dx-\int_{\R}\left[(\frac{1}{1+|x|}w)_x\right]^2dx+2\int_\R \frac{1}{(1+|x|)^2}|w_x|^2dx,\\
			J(w)=&\int_{\R}\left[(\frac{1}{1+|x|}w)_x\right]^2dx-\mu\int_\R x(1+|x|)V^{\prime}(x)\left|\frac{1}{1+|x|}w\right|^2dx.
		\end{split}
	\end{equation*}
By the similar arguments as in Lemma \ref{lem3.2}, we can confirm that
	$$H(u_1)+H(u_2)\gtrsim\|u\|^2_{H^1_\alpha}.$$
The remainder is to verify $J(w)\geq0$. Set $V_0=-x(1+|x|)V^{\prime}(x)$,  combining Lemma \ref{lemma2.1} and \eqref{eq5},  Condition $(V)$ implies that there exists a $\mu_0>0$ such that $-\frac{d^2}{dx^2}+\mu V_0$ has a unique negative eigenvalue (or no negative eigenvalue)  for all $\mu<\mu_0.$ Furthermore, the corresponding eigenfunction is even, which guarantees $J(w)\geq0$.
	
	Parallel to Sect. \ref{s3.2}, with the previous estimates for all terms in \eqref{eq21.1}, by the same arguments as those for Lemma \ref{lem4.2} and Lemma \ref{lem3.2.2}, we can derive the following conclusion. Indeed, when $K$ satisfies $(K_1)$, the smallness condition on initial data is not needed (it is the same as the proof of Lemma \ref{lem3.2.2}). We have
\begin{lemma}
Under the condition $(2)$ and the assumptions of Theorem \ref{theorem{2}}, if  $u\in C(\R;H^1_{odd}(\R))$ is a global solution of \eqref{eq20}, then for $\alpha(x)=\frac{1}{(1+|x|)^4}$,
\begin{equation*}
			-\frac{d}{dt}I(u(t))\gtrsim\|u(t)\|^2_{H^1_\alpha(\R)}.
\end{equation*}
\end{lemma}
	
\begin{proof}[Proof of \eqref{eqv2}]
It follows immediately from Lemma \ref{lem3.5.2} that, there exists a constant $M>0$ such that
		\begin{equation}\label{eq23}
			\int_{0}^{\infty} \|u(t)\|^2_{H^1_\alpha(\R)}dt\leq M.
		\end{equation}
Now, taking $\psi(x)=\frac{1}{(1+|x|)^4}$, with the help of \eqref{eq23}, and through similar arguments as in Sect. \ref{s3.2},  we can arrive at \eqref{eqv2} and thus complete the proof of Theorem \ref{theorem{2}}.
\end{proof}

\end{document}